\documentclass[11pt]{article}

\usepackage{enumerate}
\usepackage{amsthm,amsmath,amssymb}
\usepackage{graphicx}
\usepackage{lineno}
\usepackage[colorlinks=true,citecolor=black,linkcolor=black,urlcolor=blue]{hyperref}
\usepackage[english]{babel}
\usepackage{amsfonts}
\usepackage{epsfig, subfigure}
\usepackage{amscd,latexsym}
\usepackage{url}
\usepackage{color}
\usepackage{authblk}
\usepackage{subfigure}

\theoremstyle{plain}
\newtheorem{theorem}{Theorem}
\newtheorem{lemma}[theorem]{Lemma}
\newtheorem{cor}[theorem]{Corollary}
\newtheorem{prop}[theorem]{Proposition}

{\itshape}{\rmfamily}
\newtheorem{remark}[theorem]{Remark}

\newtheorem{conj}[theorem]{Conjecture}

\newtheorem{defi}[theorem]{Definition}

\theoremstyle{plain}

\newcommand\blfootnote[1]{%
	\begingroup
	\renewcommand\thefootnote{}\footnote{#1}%
	\addtocounter{footnote}{-1}%
	\endgroup
}

\begin{document}

\title{
The  Neighbor-Locating-Chromatic Number  of Pseudotrees 
}

\author[2,3]{Liliana Alcon\thanks{ liliana@mate.unlp.edu.ar}}
\author[2,3]{Marisa Gutierrez\thanks{ marisa@mate.unlp.edu.ar}}
\author[1]{Carmen Hernando\thanks{Partially supp. by projects MTM2015-63791-R(MINECO/FEDER) and Gen.Cat. DGR2017SGR1336, carmen.hernando@upc.edu}}

\author[1]{Merc\`e Mora\thanks{Partially supported by projects MTM2015-63791-R (MINECO/FEDER), Gen.Cat. DGR2017SGR1336 and H2020-MSCA-RISE project 734922-CONNECT, merce.mora@upc.edu}}
\author[1]{Ignacio M. Pelayo\thanks{Partially supported by projects MINECO MTM2014-60127-P, ignacio.m.pelayo@upc.edu}}

\affil[1]{Departament de Matem\`atiques, Universitat Polit\`ecnica de Catalunya}

\affil[2]{CMaLP, Universidad Nacional de La Plata}

\affil[3]{CONICET, Argentina}

\date{}
\maketitle

\blfootnote{\begin{minipage}[l]{0.3\textwidth} \includegraphics[trim=10cm 6cm 10cm 5cm,clip,scale=0.15]{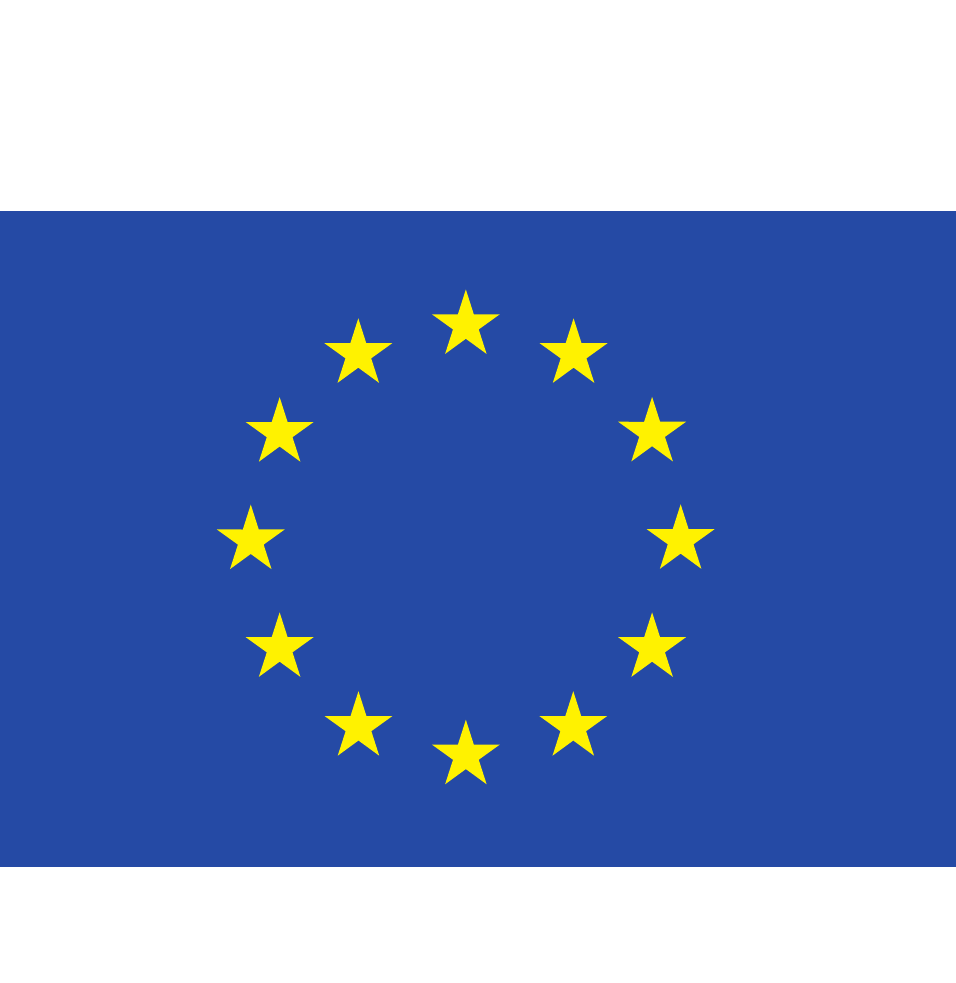} \end{minipage}  \hspace{-2cm} \begin{minipage}[l][1cm]{0.7\textwidth}
		This project has received funding from the European Union's Horizon 2020 research and innovation programme under the Marie Sk\l{}odowska-Curie grant agreement No 734922.
	\end{minipage}}

\begin{abstract}
A \emph{$k$-coloring} of a  graph $G$ is a partition of the vertices of $G$ into $k$ independent sets, which are  called \emph{colors}.
A $k$-coloring is  \emph{neighbor-locating} if any two vertices belonging to the same color can be distinguished from each other by the colors of their respective neighbors.
The \emph{neighbor-locating chromatic number} $\chi _{_{NL}}(G)$ is the minimum  cardinality of a neighbor-locating coloring of $G$.

\vspace{.02cm}
In this paper, we determine the neighbor-locating chromatic number of paths, cycles, fans and wheels.
Moreover, a procedure to construct a neighbor-locating coloring of minimum cardinality for these families of graphs is given.
We also obtain tight upper bounds on the order of trees and unicyclic graphs in terms of the  neighbor-locating chromatic number.
Further partial results for trees are also established.
\end{abstract}
\vspace{.01cm} \hspace{.95cm}{\it Key words:} coloring; location;
 neighbor-locating coloring; pseudotree.

\section{Introduction}

This work is devoted to studying a special type of vertex partitions, the so-called \emph{neighbor-locating colorings}, whithin 
the family of graphs known as \emph{pseudotrees}, that is to say, the set of all connected graphs containing at most one cycle.

There are mainly two types of location, metric location and neighbor location.
Metric-locating sets (also known as resolving sets) were introduced simultaneously in \cite{hararymelter,slater}, meanwhile 
neighbor-locating sets were  introduced in  \cite{slater2}. 
In  \cite{ChaSaZh00}, the notion of metric location was brought to the ambit of vertex partitions, and in \cite{cherheslzh02}, there were first studied  the so-called \emph{locating colorings}, i.e.,  locating partitions (also known as resolving partitions) formed by independents sets.
Both resolving partitions and locating colorings have been  extensively studied since then. 
See  e.g. \cite{chagiha08,fegooe06,ferogo14,gsrm,hemope16,royele14} and \cite{baas13,bean15,beom11,cherheslzh03,webasiut13}, respectively.

In \cite{xnl1},  we started the study of neighbor-locating  partitions formed by independent sets, which we named \emph{neighbor-locating colorings}. 
More specifically, we considered  vertex colorings such that any two vertices with the same color can be distinguished  from each other by the colors of their respective neighbors.
This  paper continues that line of research focusing on the neighbor-locating colorings of pseudotrees, i.e., of paths, cycles, trees and unicyclic graphs.

\vspace{-0.1cm}
\subsection{Basic terminology}

All the graphs considered in this paper are  connected, undirected, simple and finite. 
The vertex set and the edge set of a graph $G$ are denoted by $V(G)$ and $E(G)$, respectively. 
We let $n(G)$ be the order of $G$, i.e. $n(G)=|V(G)|$.
The \emph{neighborhood} of a vertex $v\in V(G)$ is the set $\displaystyle N(v)=\{w \in V(G):vw \in E(G)\}$. 
The \emph{degree} of $v$, defined as the cardinality of $N(v)$, is denoted by $\deg(v)$.
When $\deg(v)=1$, $v$ is called a \emph{leaf}. 
The \emph{maximum degree} $\Delta(G)$ of $G$ is defined to be $\Delta(G) = \max\{deg(v) : v \in V(G)\}$. 
The \emph{distance} between two vertices $v,w\in V(G)$ is denoted by $d(v,w)$. 
The  \emph{diameter} of $G$ is ${\rm diam}(G) = \max\{d(v,w) : v,w \in V(G)\}$.

Let  $\Pi=\{S_1,\ldots,S_k\}$ be a $k$-partition  of $V(G)$, i.e., a partition of the set of vertices of $G$ into $k$  subsets. 
If all the elements  of  $\Pi$ are independent sets, then  they are called \emph{colors} and  $\Pi$ is said to be a \emph{coloring} of $G$ (also a $k$-\emph{coloring} of $G$).
We say that a vertex $v\in V(G)$ has color $i$, or that is colored with $i$, whenever $v\in S_i$. 
   
\vspace{0.1cm}
Given a $k$-coloring $\Pi=\{S_1,\ldots,S_k\}$ of a graph $G$ and a vertex $v\in V(G)$, the  \emph{color-degree} of  $v$  is defined to be the number of different colors of $\Pi$ containing some vertex of $N(v)$, i.e., 
$ |\{ j :  N(v)\cap S_j\not= \emptyset \}| $.

\vspace{-0.2cm}
\subsection{Neighbor-locating colorings}

		A coloring  $\Pi=\{ S_1,\dots ,S_k\}$ of a graph $G$ is called a  \emph{ neighbor-locating coloring}, an \emph{NL-coloring} for short, if for every pair of  different vertices $u,v$ belonging to the same color  $S_i$, the set of colors of the neighborhood of $u$ is different from the set of colors of the neighborhood of $v$, that is,
		$$ \{ j :  N(u)\cap S_j\not= \emptyset \} \not= \{j :  N(v)\cap S_j\not= \emptyset  \}.$$
		
		The \emph{neighbor-locating chromatic number} $\chi _{_{NL}}(G)$, the \emph{NLC-number} for short,  is the minimum  cardinality of an NL-coloring of $G$.

Both neighbor-locating colorings and the neighbor-locating chromatic number of a graph were introduced in \cite{xnl1}.
As a straightforward consequence of these definitions the following remark is derived.

\begin{remark}\label{color_deg}
	Let $\Pi=\{S_1,\ldots,S_k\}$ be a  $k$-NL-coloring of a graph $G$ order $n$ and maximum degree $\Delta$. 
	For every $1\leq i \leq k$, there are at most $ \binom{k-1}{j}$ vertices in $ S_i$ of  color-degree $j$,  where $1\leq j \leq k-1$ and, consequently, 
	$\displaystyle |S_i|\le \sum_{j=1}^\Delta \binom{k-1}{j}$.
\end{remark}

\vspace{-.1cm}
\begin{theorem}[\cite{xnl1}]\label{gb}
Let $G$ be a non-trivial connected graph of order $n(G)=n$ and maximum degree $\Delta(G)=\Delta$
such that  $\chi _{_{NL}}(G)=k$. Then,
\begin{enumerate}[{\rm(1)}]
\item $\displaystyle{n\le k\, (2^{k-1}-1)}$. Moreover, this bound is tight.
\item If $\Delta\le k-1$, then $\displaystyle{n\le k\, \sum_{j=1}^{\Delta } \binom{k-1}{j}}.$
\end{enumerate}
\end{theorem}

To simplify the writing, given a $k$-NL-coloring of a graph $G$ with maximum degree $\Delta$, we denote by $a_j(k)$ the maximum number of vertices of color-degree $j$ and by $\ell(k)$, the maximum number of vertices of color-degree $1$ or $2$, where $k\ge 3$ and $1\le j\le \Delta$.
By Remark~\ref{color_deg}, we have:
\begin{itemize}
	\item $\displaystyle a_1(k)=k\cdot(k-1)$, \hspace{.2cm}
	$\displaystyle a_2(k)=\frac{k\cdot(k-1)(k-2)}{2}$.
	\item $\displaystyle \ell(k)=a_1(k)+a_2(k)=k\cdot\binom{k}{2}=\frac{k^3-k^2}{2}$.
\end{itemize}

\vspace{.2cm}
The remaining part of this paper is organized as follows. 
In Section~\ref{pn.and.cn}, the neighbor-locating chromatic number of paths and of  cycles is determined.
Section~\ref{unicyclics} deals with unicyclic graphs, providing  a tight  upper bound on the order of a unicyclic graph with a given fixed  neighbor-locating chromatic number.
In Section~\ref{s.trees}, the neighbor-locating colorings of trees is studied. Among other results, 
a tight   upper bound on the order  of a tree with a given fixed NLC-number.
Finally, in Section~\ref{op}, we summarize our results and pose some open problems.

\vspace{.3cm}
\section{Paths and Cycles}\label{pn.and.cn}

\vspace{.2cm}
This section is devoted to determine the NLC-number of all paths and cycles, i.e., all graphs with maximum degree $\Delta(G)=2$. 

\vspace{.1cm}
\begin{prop} \label{pro.ordenlk}
If $G$ is a graph of order $n$ and maximum degree  $\Delta(G)=2$ such that $\chi _{_{NL}}(G)=k\ge3$, then 
$n(G) \le \ell (k)$.
\end{prop}
\begin{proof}
	Since all vertices have color-degree at most 2, we have $n\le a_1(k)+a_2(k)=\ell (k)$.
\end{proof}

\vspace{.15cm}
\begin{cor}\label{pro.nn}
	If $G$ is either a path or a cycle  such that $n(G)>\ell (k-1)$, then $\chi_{_{NL}}(G)\ge k$.
\end{cor}
\begin{proof} This inequality directly  follows from Proposition~\ref{pro.ordenlk} taking into account that $\ell$ is an increasing function.	
\end{proof}

\vspace{.15cm}
\begin{remark} \label{pro.verticeslk}
Let  $\{S_1,\ldots,S_k\}$ be a $k$-NL-coloring of a graph $G$ of order $n=\ell (k)$ and maximum degree $\Delta(G)=2$. If $k\ge 3$, then for every $i\in \{1,\ldots,k\}$:
\begin{enumerate}[(1)]	
		\item $\displaystyle |S_i|=\binom{k}{2}$;
		
		\item there are exactly $\binom{k-1}{2}$ vertices in $S_i$ of color-degree 2; 
		
		\item there are exactly $k-1$ vertices in $S_i$ of  color-degree 1.	
\end{enumerate}
\end{remark}
\begin{proof}
Since $\Delta(G) =2$, by Remark~\ref{color_deg}  we have $|S_i|\le \binom{k-1}{1}+\binom{k-1}{2}=\binom{k}{2}$, for every $i\in \{ 1,\dots ,k\}$.
If $|S_i|<\binom{k}{2}$ for some $i\in \{ 1,\dots ,k\}$, then
 $n(G)=\sum_{j=1}^{k}|S_j|< k\binom k2=\ell (k)$, a contradiction. Hence, (1) holds.
Moreover, if, for some $i\in \{ 1,\dots ,k\}$, the number of vertices of color-degree 2 is less than $\binom{k-1}2$ or 
the number of vertices of color-degree 1 is less than $k-1$, then $|S_i|<\binom{k-1}1 + \binom{k-1}2=\binom k2$,
which contradicts (1).
\end{proof}

\vspace{.001cm}
\begin{prop}\label{pro.nolkmenos1}
    If $k\ge 3$, then $\chi_{_{NL}}(C_{\ell(k)-1})\ge k+1$.
\end{prop}
\begin{proof}
	It is easy to check that $\ell (k)-1 =\frac{k^3-k^2}{2}-1> \frac{(k-1)^3-(k-1)^2}{2}=\ell (k-1)$, if $k\ge 3$. Hence, by Corollary~\ref{pro.nn}  we have $\chi_{_{NL}}(C_{\ell(k)-1})\ge k$.
    
    Suppose that, on the contrary, $\chi_{_{NL}}(C_{\ell(k)-1}) = k$ and consider a $k$-NL-coloring  $\{S_1,\dots, S_k\}$ of  $C_{\ell(k)-1}$.
    Similarly as argued in the proof of Remark~\ref{pro.verticeslk}, there must be exactly $\binom k2$ vertices of all but one of the $k$ colors, and $\binom k2-1$ vertices of the remaining color. We may assume without loss of generality that 
    $|S_k|=\binom k2-1$ and 
    $|S_i|=\binom k2$,
    whenever $i\neq k$.
    To attain this number of vertices, for each color $i\in \{1,\dots , k-1 \}$, there must be a vertex  of color-degree 1 in $S_i$ with both neighbors in $S_k$ and, for every $j\not= i,k$, there must be a vertex of color-degree 2 in $S_i$ with a neighbor in $S_k$ and the other in $S_j$.
    Besides, both neighbors of a vertex of $S_{k}$ belong to $S_1\cup \dots \cup S_{k-1}$.
    Hence, if we sum the number of neighbors colored with $k$ for all the vertices belonging to  $S_1\cup \dots \cup S_{k-1}$, we count exactly twice each vertex of $S_k$.
    Therefore, $2|S_k|=(k-1) (2+(k-2))=k(k-1)$, contradicting that $|S_k|=\binom k2-1=\frac{k(k-1)}2-1$.
\end{proof}

Next theorem establishes the NLC-number of paths and cycles of small order.

\begin{theorem}\label{pro.base} The values of the NLC-number for paths and cycles of order at most $9$ are:

\begin{enumerate}[\rm (1)]
\item  $\chi_{_{NL}}(P_{2})=2$.
\item  $\chi_{_{NL}}(P_{n})=3$, if $3\le n\le 9$.
\item  $\chi_{_{NL}}(C_{n})=3$, if $n\in \{ 3,5,7,9\}$.
\item   $\chi_{_{NL}}(C_{n})=4$, if $n\in \{ 4,6,8 \}$
\end{enumerate}

\end{theorem}
\begin{proof}
    Trivially, $P_2$ is the only graph $G$ with $\chi_{_{NL}} (G)=2$. According to  Proposition~\ref{pro.ordenlk}, the order of a path or a cycle with NLC-number equal to 3  is at most 9. 
    For $P_n$ with $3\le n\le 9$ and for $C_n$ with $n\in \{  3,5,7,9 \}$, a 3-NL-coloring is displayed in Figure~\ref{colorsmallcycles}.

    \begin{figure}[!hbt]
        \begin{center}
            \includegraphics[width=0.6\textwidth,page=1]{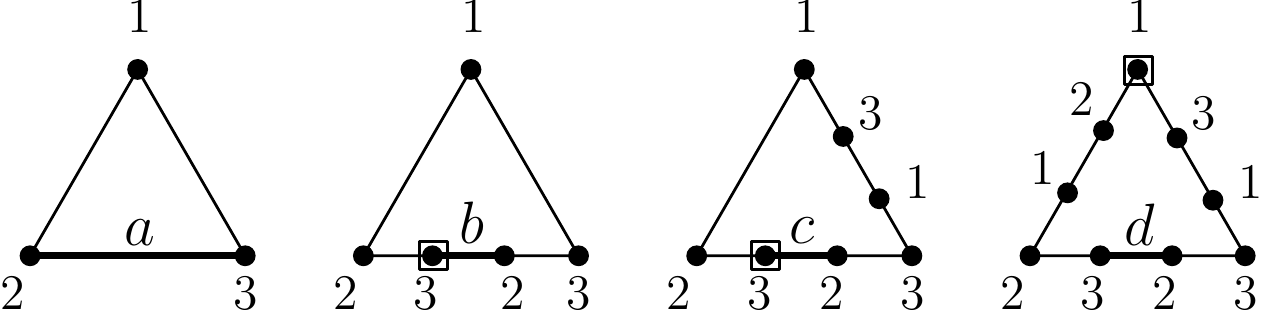}
            \caption{From left to right, a 3-NL-coloring of  cycles $C_{3}$, $C_{5}$, $C_{7}$ and $C_{9}$.
                A 3-NL-coloring of  paths $P_{4}$, $P_{6}$ and $P_{8}$ can be obtained by removing the squared vertices
                and a 3-NL-coloring of  paths $P_{3}$, $P_{5}$, $P_{7}$ and $P_{9}$ can be obtained by removing the edges $a$, $b$, $c$ and $d$, respectively.}\label{colorsmallcycles}
        \end{center}
    \end{figure}


    Clearly, $\chi_{_{NL}}(C_{4})=4$ since, as was proved in \cite{xnl1}, for every complete multipartite graph $G$ of order $n$, $\chi_{_{NL}}(G)=n$.
    
    An exhaustive analysis of all possible cases shows that $\chi_{_{NL}} (C_6)\ge 4$. Clearly, a 4-NL-coloring for $C_6$ can be obtained by inserting a vertex colored with 4 in any edge of the 3-NL-coloring given for $C_5$ in Figure~\ref{colorsmallcycles}. Hence, $\chi_{_{NL}}(C_{6})=4$.
    
 By Proposition~\ref{pro.nolkmenos1}, $\chi_{_{NL}} (C_8)\ge 4$.
 Clearly, a 4-NL-coloring for $C_8$ can be obtained by inserting a vertex colored with 4 in any edge of the 3-NL-coloring given for $C_7$ in Figure~\ref{colorsmallcycles}. Hence, $\chi_{_{NL}}(C_{8})=4$.
\end{proof}

In what follows, several technical results are  given. 
They are needed to prove Theorem~\ref{thm.chi_CiclosCaminos}, where the  NLC-number of all paths and cycles of  order greater than 9 is determined.

\begin{lemma}\label{pro.deg1deg2}
In any NL-coloring of the cycle $C_n$, every vertex of color-degree $1$ has at least one  neighbor of color-degree 2.
\end{lemma}
\begin{proof} 
Let $x$ be a vertex of color-degree 1. 
In order to derive a contradiction, suppose that its two neighbors $y$ and $z$ have  also color-degree 1.
Then, $y$ and $z$ have the same color, and each one of them  has its two neighbors with the same color, indeed the color of $x$.
A contradiction in any NL-coloring.
\end{proof}

The following operations will be used to obtain NL-colorings of some cycles from NL-colorings of smaller cycles by inserting vertices of degree 2.

\begin{defi}\label{pro.insertarpar}
    {\rm Consider a  coloring (not necessarily neighbor-locating) of a cycle $C_n$.
    Let  $x$ and $y$ be a pair of adjacent vertices colored with $i$ and $j$, with $i\not= j$, respectively.
    The following operations produce a coloring of a cycle of order $n+1$ and $n+2$, respectively.}

    \begin{enumerate}[{\rm (OP1)}]
    	 \item  {\rm If $x$ and $y$ have color-degree 1, insert a new vertex $z$  colored with $h$, $h\not= i,j$,  in the edge $xy$
    	   	(see Figure~\ref{fig.lematecnic}, left).} 
    	\item  
    	{\rm If $x$ and $y$ are vertices of color degree 2, insert two new vertices $x'$ and $y'$ in the edge $xy$, so that $xx'$, $x'y'$ and $y'y$ are edges of the new cycle,
    and  $x'$ and $y'$ are colored with $j$ and $i$, respectively (see Figure~\ref{fig.lematecnic}, right).}  
    \end{enumerate}
    
    \begin{figure}[!hbt]
        \begin{center}
            \includegraphics[width=0.80\textwidth,page=2]{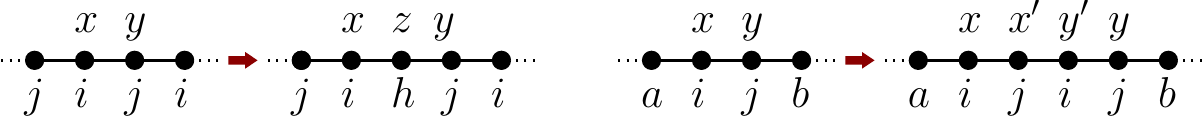}
            \caption{Illustrating Definition~\ref{pro.insertarpar}. Left, $|\{i,j,h\}|=3$  and right, $|\{i,j,a,b\}|=4$.}\label{fig.lematecnic}
        \end{center}
    \end{figure}
      
\end{defi}
Observe that the colors of vertices $x$ and $y$ are preserved with these operations. Operation (OP2) preserves also their color-degree and the set of colors of their neighbors. However, operation (OP1) changes the color-degree of $x$ and $y$ from $1$ to  $2$. Besides, the vertex $z$ added by operation (OP1) has color degree 2, meanwhile the vertices $x'$ and $y'$ added by operation (OP2) have color degree 1.
Notice that the color, color-degree and set of colors of the neighborhood of any other vertex different from $x,y,z$ when applying operation (OP1), and different from  $x,y,x'y'$, when applying operation (OP2), remain unchanged.

\newpage
The following type of NL-colorings will play an important role to construct $NL$-colorings of paths and cycles.
\begin{defi}
An NL-coloring is said to be 1-paired if every vertex of color-degree $1$ has a neighbor of color-degree $1$. 
\end{defi}

\begin{remark}\label{1pairedC9}
	If a $k$-NL-coloring is 1-paired, then every vertex of color-degree 2 has at least one neighbor of color degree 2.
\end{remark}

\begin{remark}\label{1pairedC9}
	The $3$-NL-coloring given in Figure~\ref{colorsmallcycles} for the cycle $C_9$ is 1-paired. 
\end{remark}

\begin{lemma}\label{pro.1pairedcoloring}
    Let $k\ge 4$ be an integer.
    Then,

\begin{enumerate}[\rm (1)]
    
  \item for every $n\in \{\ell (k-1)+1, \cdots, \ell(k)-2,\ell (k)\}$, there is  a 1-paired $k$-NL-coloring of $C_n$.
    
  \item If $n\neq a_2(k)$,
  then there is a 1-paired $k$-NL-coloring of $C_n$ containing (at least) a pair of adjacent vertices of color-degree 1. 
    
  \item If $n= \ell(k)$, then there is a 1-paired $k$-NL-coloring of $C_n$ containing a sequence of 7 consecutive vertices colored  with $1,2,1,2,3,2,3$, respectively.
    
 \end{enumerate}
\end{lemma}
\begin{proof} 
	Let $k\ge 4$. We begin by proving that the stated result is true if there exists a 1-paired $(k-1)$-NL-coloring of $C_{\ell (k-1)}$ (an example of the procedure described below is shown in Figure~\ref{colorcicle} for $k=4$).
	
	Suppose that $\{S_1,\dots ,S_{k-1}\}$ is a 1-paired $(k-1)$-NL-coloring of $C_{\ell (k-1)}$.
	As a consequence of Remark~\ref{pro.verticeslk}, by the one hand, if $i,j,h$ are different colors from $\{1,\dots , k-1\}$, then there must be a vertex in $S_h$ of color-degree 2 with a neighbor in $S_i$ and the other in $S_j$. On the other hand,
	for each pair of distinct colors $i,j\in \{ 1,\dots ,k-1\}$, there must be a vertex in $S_i$ of color-degree 1 with both neighbors in $S_j$. In this last case, since the coloring is 1-paired, one of these neighbors must have color-degree 1.
Therefore, for each pair of distinct colors $i,j\in \{ 1,\dots ,k-1\}$,
there is a pair of adjacent vertices $x\in S_i$ and $y\in S_j$ of color-degree 1.
By Lemma~\ref{pro.deg1deg2}, these $\binom{k-1}2$ pairs of adjacent vertices of color-degree 1 are pairwise disjoint.


For every one of the $\binom{k-1}2$ pairs of adjacent vertices of color-degree 1 we can insert a new vertex colored with a new color $k$ as described in (OP1). Note that in this way, we add at each step a new vertex of color-degree 2 to $S_k$ (the set of vertices with the new color $k$), and there is a pair of adjacent vertices of color-degree 1 that become vertices of color-degree 2. Besides, by construction, we have a $k$-NL-coloring at each step. Therefore, after $\binom{k-1}2$ steps we obtain a $k$-NL-coloring of the cycle of order $\ell(k-1)+\binom{k-1}2=a_2(k)$ such that there are no vertices of color-degree 1.
Moreover, at every intermediate step we have a 1-paired $k$-NL-coloring of the corresponding cycle.
In particular, we obtain 1-paired $k$-NL-colorings of $C_{a_2(k)-1}$ and of $C_{a_2(k)}$ such that for each unordered pair $\{ i,j \} \subseteq \{1,\dots ,k\}$ there exists a pair of adjacent vertices $x\in S_i$ and $y\in S_j$ of color-degree 2. Besides, the 1-paired NL-coloring of $C_{a_2(k)}$ has no vertices of color degree 1, while the 1-paired NL-coloring of $C_{a_2(k)-1}$ has exactly one pair of adjacent vertices of color-degree 1.

Now, starting with the $k$-NL-coloring obtained for $C_{a_2(k)}$, choose an edge with endpoints of  color-degree 2 in $S_i$ and $S_j$, respectively, for every pair $i,j$ of distinct colors of $\{1,\dots ,k\}$. 
By successively applying (OP2) to the $\binom k2$ edges chosen in this way, it is possible to add up to $\binom {k} 2$ pairs of adjacent vertices of color-degree 1 giving rise to a 1-paired coloring of $C_{n}$, whenever $n$ has the same parity as $a_2(k)$, and $a_2(k)\le n\le a_2(k)+2\binom{k}2=\ell (k)$.

We can proceed in a similar way starting with the 1-paired $k$-NL-coloring of $C_{a_2(k)-1}$. 
The difference with respect to the preceding case is that now we already have a pair of vertices of color-degree 1 that we may assume are colored with $i'$ and $j'$, respectively, with their neighbors in $S_{j'}$ and in $S_{i'}$, respectively. Hence, in order to have an NL-coloring at each step, we don't choose any edge with endpoints of color-degree 2 and colored with $i'$ and $j'$. By successively applying (OP2) to the $\binom k2-1$ chosen edges, it is possible to add up to $\binom {k} 2-1$ pairs of adjacent vertices of color-degree 1 obtaining a 1-paired coloring of $C_{n}$, whenever $n$ has the same parity as $a_2(k)-1$, and $a_2(k)-1\le n\le a_2(k)-1+2\left( \binom{k}2 -1\right)=\ell (k)-2$.

This procedure gives a 1-paired $k$-NL-coloring of $C_n$, whenever $\ell(k-1)<n\le \ell (k)$ and $n\not= \ell (k)-1$.
Moreover, by construction, the obtained 1-paired $k$-NL-coloring of $C_n$ has at least a pair of adjacent vertices of color-degree 1, except for the case $n=a_2(k)$, that has no vertex of color-degree 1.
The sequence of colors $1,2,1,2,3,2,3$ can be obtained in the following way.
Consider the edges incident to a vertex $u\in S_2$ and with neighbors colored with 1 and 3 in $C_{a_2(k)}$ (we know that it exists) and begin applying (OP2) to the edges incident to $u$.

Now we proceed to prove the stated result by induction. 
For $k=4$, we have $\ell (3)=9$ and a 1-paired $3$-NL-coloring of $C_9$ is given in Figure~\ref{colorcicle}. Hence, using the procedure described above, we have that the stated result is true for $k=4$.  
Now let $k>4$. By induction hypothesis, the stated result is true for $k-1$,
that is, there exists a 1-paired $(k-1)$-NL-coloring of $C_{\ell (k-1)}$ and we can proceed as described above to demonstrate the result for $k$. 
  \end{proof}
    
    \vspace{-.2cm}
    \begin{figure}[!hbt]
        \begin{center}
            \includegraphics[width=0.74\textwidth,page=3]{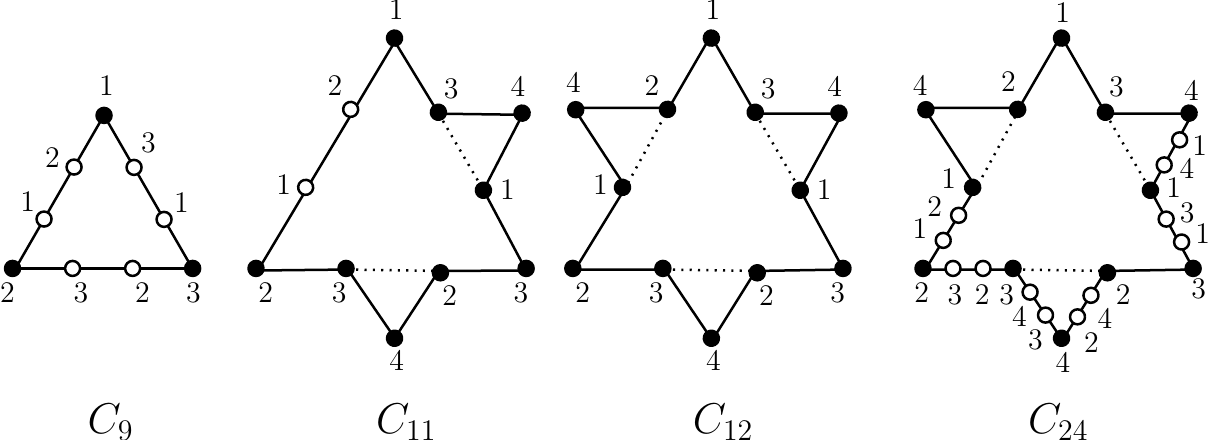}]
            \caption{Obtaining a 4-NL-coloring from a 1-paired $3$-NL-coloring of the cycle $C_9$. 
            	In white, the vertices of color-degree 1. 
            	Recall that $\ell(3)=9$, $\ell(4)=24$, $a_2(4)-1=11$ and $a_2(4)=12$. 
            	Inserting vertices of color-degree 2 in some edges of $C_{9}$, we achieve 4-NL-colorings of $C_{n}$, whenever $n\in \{10,11,12\}$. 
            	Inserting pairs of white vertices in some edges of $C_{11}$ and of $C_{12}$ we achieve 4-NL-colorings of $C_n$, whenever $n\in \{13,\dots ,24\}\setminus \{23\}$.}\label{colorcicle}
        \end{center}
    \end{figure}

\vspace{-.7cm}
\begin{lemma} \label{xnlpn}
Let $k\ge 4$ be an integer. 
If there is a 1-paired $k$-NL-coloring of $C_{n}$, then there is a  $k$-NL-coloring of $P_n$.
\end{lemma}
\begin{proof}
    Consider one of the 1-paired $k$-NL-colorings of $C_{n}$ described to prove Lemma~\ref{pro.1pairedcoloring}.
    For $n\not= a_2(k)$, it is enough to remove from the cycle any edge joining two adjacent vertices of color-degree 1.
    For $n = a_2(k)$, the removal of any edge $xy$ gives rise to only two vertices, $x$ and $y$, of color-degree 1, and
    the set of colors of the neighborhood of any other vertex is not modified. 
    Hence, in any case, we have a $k$-NL-coloring of the path $P_n$.
    \end{proof}

\newpage
\begin{lemma} \label{xnlcaminoellk-1} 
For every integer $k\ge 4$, there is a $k$-NL-coloring of the path $P_{\ell(k)-1}$.
\end{lemma}
\begin{proof}
    Consider the $k$-NL-coloring of the cycle $C_{\ell(k)}$ containing the sequence of vertices colored with $1,2,1,2,3,2,3$
described in Lemma~\ref{pro.1pairedcoloring}.
If we remove  from the preceding sequence the vertex colored with $2$ whose neighbors have colors 1 and 3, respectively, then we obtain a $k$-NL-coloring of $P_{\ell(k)-1}$.
\end{proof}

\begin{lemma}\label{xnlcicloellk-1} 
For every integer $k\ge 4$, there is a $(k+1)$-NL-coloring of the cycle $C_{\ell(k)-1}$.
\end{lemma}
\begin{proof}
    Consider a $k$-NL-coloring of the cycle $C_{\ell(k)-2}$. Let $x$ and $y$ be adjacent vertices. 
    Remove the edge $xy$ and add a new vertex $z$  adjacent to $x$ and $y$. 
    If $z$ is colored with a new color $k+1$, then we have a $(k+1)$-NL-coloring of $C_{\ell(k)-1}$.
\end{proof}

As a consequence of Corollary \ref{pro.nn} and Lemmas \ref{pro.1pairedcoloring}, \ref{xnlpn}, \ref{xnlcaminoellk-1} and  \ref{xnlcicloellk-1}, we can determine the neighbor-locating chromatic number of graphs and cycles of order at least $4$. 
Notice that the given proofs of these lemmas are constructive. 
Hence, it is possible to produce NL-colorings of minimum cardinality for all paths and cycles.

\begin{theorem}\label{thm.chi_CiclosCaminos} 
Let $k,n$  be integers such that $k\ge 4$ and $\ell (k-1) <n\le \ell (k)$. 
Then,

\begin{enumerate}[\rm (1)]
    \item $\chi_{_{NL}}(P_{n})=k$.

    \item $\chi_{_{NL}}(C_{n})=k$, if  $n\not= \ell (k)-1$.

    \item $\chi_{_{NL}}(C_{n})=k+1$, if  $n=\ell (k)-1$.
\end{enumerate}  
  
\end{theorem}

\section{Fans and Wheels}\label{fa.and.wh}

The graphs that are obtained by adding a new vertex adjacent to every vertex of either the path or the cycle of order $n-1$
are the \emph{fan} and the \emph{wheel} of order $n$, that are denoted by $F_n$ and $W_n$, respectively.
The preceding theorem allows us to determine the NLC-number of these graphs.

\begin{lemma}\label{univvertex} 
	If  $G'$ is the graph obtained from a graph $G$ by adding a new vertex adjacent to every vertex of $G$, then
	 $\chi_{_{NL}} (G')=\chi_{_{NL}} (G)+1$.
\end{lemma}
\begin{proof}
Let	$\chi_{_{NL}} (G)=k$, $\chi_{_{NL}} (G')=k'$  and let $u$ be the vertex of $G'$ adjacent to every vertex of $G$.
Obviously, $k'\le k+1$, because a $(k+1)$-NL-coloring of $G'$ can be obtained from a 
$k$-NL-coloring of $G$ by assigning a new color to vertex $u$.
On the other hand, if we have a $k'$-NL-coloring $\Pi'$ of $G'$, then the color assigned to $u$ must be different from the color assigned to any vertex of $G$. Moreover, since every vertex of $G$ is adjacent to $u$, the coloring $\Pi'$ restricted to the vertices of $G$ is a $(k'-1)$-NL-coloring of $G$, implying that $k\le k'-1$.
Hence, $k'=k+1$.
\end{proof}

\begin{theorem}\label{t:fan_wheel_small}
The NLC-number of fans and wheels of order $n$, $4\le n\le 10$, is:
	
	\begin{enumerate}[\rm (1)]
		\item  $\chi_{_{NL}}(F_{n})=4$, if $4\le n\le 10$.
		\item  $\chi_{_{NL}}(W_{n})=4$, if $n\in \{ 4,6,8,10\}$.
		\item   $\chi_{_{NL}}(W_{n})=5$, if $n\in \{ 5,7,9 \}$
	\end{enumerate}	
\end{theorem}
\begin{proof} 
	It is a direct consequence of Lemma~\ref{univvertex} and Theorem  
	\ref{pro.base}.
\end{proof}

\newpage
\begin{theorem}\label{t:fan_wheel}
Let $k,n$  be integers such that $k\ge 4$ and  $\ell (k-1)+1 <n\le \ell (k)+1$. 
Then,

\begin{enumerate}[\rm (1)]
 \item $\chi_{_{NL}} (F_n)= k+1$.

 \item $\chi_{_{NL}} (W_n)= k+1$, if $n\neq \ell (k)$.

 \item $\chi_{_{NL}} (W_n)= k+2$, if $n= \ell (k)$.
 \end{enumerate}
 
\end{theorem}
\begin{proof} 
It is a direct consequence of Lemma~\ref{univvertex} and Theorem  \ref{thm.chi_CiclosCaminos}.
\end{proof}

Notice that NL-colorings of minimum cardinality for fans and wheels can be constructed from NL-colorings of paths and cycles, respectively, by assigning a new color to the added vertex.

\section{Unicyclic graphs}\label{unicyclics}

A connected graph is called \emph{unicyclic} if it  contains precisely one cycle.

\begin{theorem}\label{unicyclic}
Let $G$ be a  unicyclic graph. 
If $\chi _{_{NL}}(G)=k\ge3$, then $$\displaystyle{n(G)\le 2a_1(k)+a_2(k)=\,\,\, \frac 12 (k^3+k^2-2k)}.$$

Moreover, if the equality holds, then $G$ has maximum degree $3$, and it contains  $k(k-1)$ leaves,  $\frac{k(k-1)(k-2)}{2}$ vertices of degree 2, and $k(k-1)$ vertices of degree 3.
\end{theorem}
\begin{proof}
Let $n$, $n_1$, $n_2$ and $n_{_{\ge 3}}$ be respectively the order, the number of leaves, the number of vertices of degree $2$  and the number of vertices of degree at least $3$ of $G$.
On the one hand, we know that
    \begin{align*}
    n_1+2\, n_2+ \sum_{\deg(u)\ge 3} \deg (u)= \sum_{u\in V(G)} \deg (u)=2 | E(G) | =2 \, n =2 (n_1+n_2+ n_{_{\ge 3}} ).
    \end{align*}
    From here, we deduce that
    \begin{align}\label{n3unicyclic}
    n_1=\sum_{\deg (u)\ge 3}{(\deg(u)-2)}\ge n_{_{\ge 3}} .
    \end{align}
    On the other hand,  $\chi _{_{NL}}(G)=k$ implies  
    $n_1\le k(k-1)$ and  $n_2\le k \binom{k-1}{2} $. Therefore,  

    \begin{align*}n&=n_1+n_2+n_{_{\ge 3}}\\
    &\le  k\Big( (k-1)+\binom{k-1}{2}\Big) + n_1\\
    &\le    k\Big( (k-1)+\binom{k-1}{2}\Big)+ k(k-1)  \\
    &= 2a_1(k)+a_2(k)\\
    &= \frac 12 (k^3+k^2-2k).
    \end{align*}

Now, assume that there is a unicyclic graph $G$  attaining this bound. 
In such a case, the inequalities in the preceding expression must be equalities.
Thus, $n_{_{\ge 3}}=n_1=k(k-1)$, and  $n_2=k\binom{k-1}{2}=\frac{k(k-1)(k-2)}{2}$.
Finally, from Inequality~\ref{n3unicyclic}, we deduce that $n_{_{\ge 3}}=n_1$ if and only if there are no vertices of degree greater than $3$.
Therefore, there are exactly $n_{_{\ge 3}}$ vertices of degree $3$. 
Since $n_{_{\ge 3}}=n_1=k(k-1)$, the proof is complete.
\end{proof}

The bound given in Theorem~\ref{unicyclic} is tight for $k\ge 5$. 
To prove this, we first give a $k$-NL-coloring of the comb of order $2k(k-1)$.
Recall that, for every integer $m\ge3$,  the \emph{comb} $B_{m}$ is the tree obtained by attaching one leaf at every vertex of $P_m$, the path of order $m$.


\begin{prop}\label{pro.colorcomb} 
For every $k\ge 5$, there is a $k$-NL-coloring of the comb $B_{k(k-1)}$. 
\end{prop}
\begin{proof}
    Let $k\ge 5$. 
    Consider the comb $B_{k(k-1)}$ obtained by hanging a leaf to each vertex of a path $P$ of order $k(k-1)$.
    We color with color $1$ the leaves hanging from the first $k-1$ vertices of the path $P$; with color $2$ the leaves hanging from the following $k-1$ vertices of $P$; and so on.
    For every $r\in \{ 1,\dots ,k \}$, consider the set $M_r$ containing the $k-1$ vertices of $P$ adjacent to the leaves colored with $r$. 
    We  define a bijection between the vertices of $M_r$ and the $k-1$ colors of the set $L_r=\{ 1,2,\dots ,k\}\setminus \{ r \}$.
    Set $M_r=\{ x_1^r,\dots ,x_{k-1}^r\}$ so that $x_i^r x_{i+1}^r\in E(P)$ for every $i\in \{ 1,\dots ,k-2 \}$, and
    $x_{k-1}^rx_1^{r+1}\in E$ if $r<k$.
\medskip
    
    \begin{figure}[!ht]
    	\begin{center}
    		\includegraphics[width=0.57\textwidth,page=4]{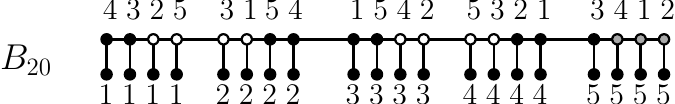}
    		\vspace{4mm}
    		\newline
    		\includegraphics[width=0.75\textwidth,page=5]{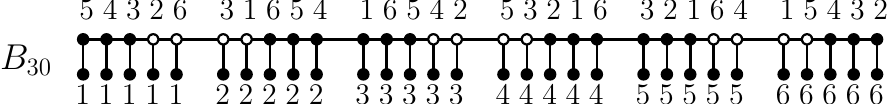}
    		\vspace{4mm}
    		\newline
    		\includegraphics[width=0.95\textwidth,page=6]{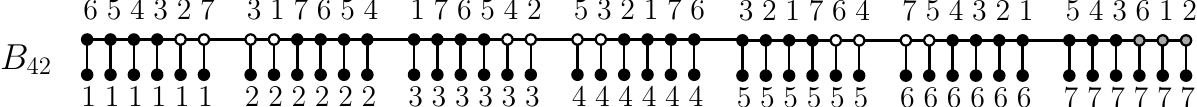}
    		\vspace{5mm}
    		\newline
    		\includegraphics[width=0.75\textwidth,page=7]{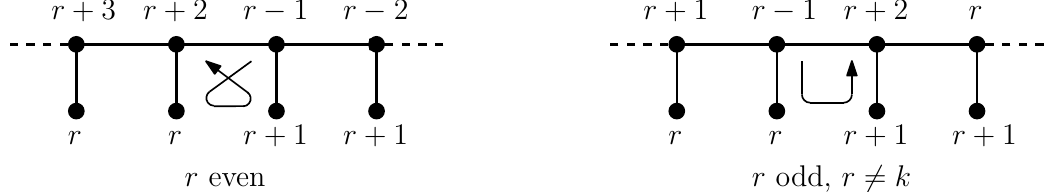}
    		\caption{A 5-NL-coloring of  the comb $B_{20}$, a 6-NL-coloring of  the comb $B_{30}$ and a 7-NL-coloring of  the comb $B_{42}$.
    			In white, adjacent vertices of $M_r$ with no consecutive colors modulo $k$.
    			In $B_{20}$ and in $B_{42}$, we have shifted the colors of the vertices in gray with respect to the general rule used to the vertices of $M_r$,
    			when $r$ is odd.    Below, the general rule for coloring adjacent vertices of consecutive groups $M_r$ and $M_{r+1}$ and the leaves hanging from them. In all cases, the colors involved are  $r-1$, $r$, $r+1$ and $r+2$.}\label{comb6}
    	\end{center}
    \end{figure}
 
 \newpage
    We assign the colors of $L_r$ in cyclically decreasing order to the vertices $x_1^r,\dots ,x_{k-1}^r$ beginning with different colors in each case:
    \begin{enumerate}[$\bullet$]
    	\item If  $r$ is even, then we begin with $r+1$ modulo $k$. 
    	Therefore, $x_1^r$ and  $x_2^r$  are colored respectively with $r+1$ and $r-1$ modulo $k$.
    	\item If  $r$ is odd and $r<k$, then we begin with $r-2$ modulo $k$.
    	Therefore, $x_{k-2}^r$  and $x_{k-1}^r$ are colored respectively with $r+1$ and $r-1$ modulo $k$.
    	\item If $r$ is odd and $r=k$, then  we proceed as in the case $r$ odd and $r<k$,
    	but we switch the colors of the last three vertices so that $x_{k-3}^k$, $x_{k-2}^k$ and $x_{k-1}^k$
    	have color $k-1$, $1$ and $2$, respectively.
    \indent
    \end{enumerate}
    
        See the defined $k$-NL-coloring of the comb $B_{k(k-1)}$ for $k\in \{5,6,7 \}$ in Figure~\ref{comb6}.
    \medskip
          
    Notice that the colors of two consecutive vertices of $M_r$ differ by one unit modulo $k$, except for the first two vertices, when $r$ is even, and for the last two vertices, when $r$ is odd.
    Besides,
    the first vertex of $M_r$ is always colored with an odd number and the last vertex of $M_r$ is colored with an even number whenever $k$ is even or when $k$ is odd and $r\notin \{1,k-1,k\}$.
We claim that this procedure gives a $k$-NL-coloring of the comb $B_{k(k-1)}$.
\medskip

We only have to prove that for every pair of non-leaves with the same color, the sets of colors of their neighborhoods are different.

Let $l\in \{ 1,\dots ,k\}$.
There are exactly $k-1$ non-leaves colored with $l$, and exactly one of them belongs to $M_r$, for every $r\in \{1,\dots ,k\}\setminus \{ l \}$.
Let $v_l^r$ be the vertex of $M_r$ colored with $l$.
Notice that the colors of the neighbors of $v_l^r$ are $\{ r, l-1,l+1\}$, except when $v_l^r$ occupies the first or last positions in $M_r$. Concretely, this happens  for $r\in \{ l-2,l-1,l+1\}$, if $l$ is even, and for $r\in \{ l-1,l+1,l+2\}$, if $l$ is odd,
whenever $l\not= \{ 1,2,3,k-1,k\}$. 
Those last cases are analyzed separately.
\medskip

We summarize in  Table 1 the colors of the neighbors of $v_l^r$, $r\not= l$, for all cases.
Observe that the sets of colors of the neighbors of $v_l^1,v_l^2,\dots ,v_l^{k-1}$ are  different for every case, so that we have a $k$-NL-coloring for each case.
\end{proof}


    \newpage
        \noindent
        \begin{tabular}{ | l | l |}
            \hline
            \multicolumn{2}{|c|}{ $l$ even, $l\not= 2,k-1,k$}\\
            \hline
            $ r $ & colors of $N(v_l^r)$ \\
            \hline
            \hline
            $r\notin \{ l-2, l-1, l+1 \}$ & $\{ r,l-1,l+1\}$  \\
            \hline
            $l-1$ & $\{ l-2,l-1,l+1 \} $ \\
            \hline
            $l+1$  & $\{ l+1,l+2,l+3 \} $\\
            \hline
            $l-2$  & $\{ l-3, l-2,l+1 \} $\\
            \hline
        \end{tabular}
        %
        %
        \begin{tabular}{ | l | l |}
            \hline
            \multicolumn{2}{|c|}{ $l$ odd, $l\not= 1,3,k-1,k$}\\
            \hline
            $ r $ & colors of $N(v_l^r)$\\
            \hline
            \hline
            $r\notin \{ l-1, l+1, l+2 \}$ & $\{ r,l-1,l+1\}$ \\
            \hline
            $l-1$ & $\{ l-3,l-2,l-1 \}$ \\
            \hline
            $l+1$ & $\{ l-1, l+1,l+2 \} $ \\
            \hline
            $l+2$ & $\{ l-1, l+2,l+3 \}$ \\
            \hline
        \end{tabular}
        \vspace{3mm}

        \noindent
        \begin{tabular}{ | l | l |}
            \hline
            \multicolumn{2}{|c|}{ $l=1$}\\
            \hline
            $ r $ \phantom{xxxxxxxxxxxxxxxx} & colors of $N(v_1^r)$ \phantom{xxx}\\
            \hline
            \hline
            $r\notin \{ 2,3,4 \}$ & $\{ r,2,k\}$  \\
            \hline
            2 & $\{ 2,3,k\} $ \\
            \hline
            3  & $\{ 3,4,k \} $\\
            \hline
            $k$ even  & $ \{ k-2,k-1,k \}$\\
            $k$ odd  & $ \{ 2, k-2,k-1 \}$\\
            \hline
        \end{tabular}
        %
        \noindent
        \begin{tabular}{ | l | l |}
            \hline
            \multicolumn{2}{|c|}{ $l=2$}\\
            \hline
            $ r $ \phantom{xxxxxxxxxxxxxxxx} & colors of $N(v_2^r)$ \phantom{xxx} \\
            \hline
            \hline
            $r\notin \{ 1,3,k \}$ & $\{ r,1,3\}$  \\
            \hline
            1 & $\{ 1,3,k\} $ \\
            \hline
            3  & $\{ 3,4,5 \} $\\
            \hline
            $k$ even  & $ \{ 3,k \}$\\
            $k$ odd  & $ \{ 1,k \}$\\
            \hline
        \end{tabular}
        \vspace{3mm}

        \noindent
        \begin{tabular}{ | l | l |}
            \hline
            \multicolumn{2}{|c|}{ $l=3$, $k\ge 6$ even}\\
            \hline
            $ r $ & colors of $N(v_3^r)$  \phantom{xxxx} \\
            \hline
            \hline
            $r\notin \{ 2,4,5,k-1,k \}$ & $\{ r,2,4\}$  \\
            \hline
            2 & $\{ 1,2,k\} $ \\
            \hline
            4  & $\{ 2,4,5 \} $\\
            \hline
            5 & $\{ 2,5 ,6\} $\\
            \hline
            $k-1$  & $ \{ 2,4,k-1 \}$\\
            \hline
            $k$   & $ \{ 2,4,k \}$\\
            \hline
        \end{tabular}
        \begin{tabular}{ | l | l |}
            \hline
            \multicolumn{2}{|c|}{ $l=3$, $k\ge 7$ odd}\\
            \hline
            $ r $ & colors of $N(v_3^r)$ \phantom{xxxx}\\
            \hline
            \hline
            $r\notin \{ 2,4,5,k-1,k \}$ & $\{ r,2,4\}$  \\
            \hline
            2 & $\{ 1,2,k\} $ \\
            \hline
            4  & $\{ 2,4,5 \} $\\
            \hline
            5 & $\{ 2,5 ,6\} $\\
            \hline
            $k-1$& $ \{ 2,k-1,k \}$\\
            \hline
            $k$   & $ \{ 4,k-1,k \}$\\
            \hline
        \end{tabular}
        \vspace{3mm}

        \noindent
        \begin{tabular}{ | l | l |}
            \hline
            \multicolumn{2}{|c|}{ $l=k-1$, $k$ even}\\
            \hline
            $ r $  \phantom{xxxxxxxxxxxxxxxx} & colors of $N(v_{k-1}^r)$ \\
            \hline
            \hline
            $r\notin \{ 1,k-2,k \}$ & $\{ r,k-2,k\}$  \\
            \hline
            $1$ & $\{ 1,k-2 \} $ \\
            \hline
            $k-2$ & $\{k-4,k-3,k-2 \} $ \\
            \hline
            $k$  & $\{ 1,k-2,k\} $\\
            \hline
        \end{tabular}
        %
        %
        \noindent
        \begin{tabular}{ | l | l |}
            \hline
            \multicolumn{2}{|c|}{ $l=k-1$, $k$ odd}\\
            \hline
            $ r $ \phantom{xxxxxxxxx} & colors of $N(v_{k-1}^r)$ \phantom{}\\
            \hline
            \hline
            $r\notin \{ 1,k-3,k-2,k \}$ & $\{ r,k-2,k\}$  \\
            \hline
            $1$ & $\{ 1,k-2 \} $ \\
            \hline
            $k-3$ & $\{k-4,k-3,k \} $ \\
            \hline
            $k-2$ & $\{k-3,k-2,k \} $ \\
            \hline
            $k$  & $\{ 1,3,k\} $\\
            \hline
        \end{tabular}
        \vspace{3mm}

        \noindent
        \begin{tabular}{ | l | l |}
            \hline
            \multicolumn{2}{|c|}{ $l=k$, even}\\
            \hline
            $ r $   \phantom{xxxxxxxxxx} & colors of $N(v_{k}^r)$  \phantom{xxxx} \\
            \hline
            \hline
            $r\notin \{ 1,k-2,k-1 \}$ & $\{ r,1,k-1\}$  \\
            \hline
            $1$ & $\{ 1,2,3\} $ \\
            \hline
            $k-2$ & $\{1,k-3,k-2 \} $ \\
            \hline
            $k-1$  & $\{ 1,k-2,k-1\} $\\
            \hline
        \end{tabular}
        %
        \noindent
        \begin{tabular}{ | l | l |}
            \hline
            \multicolumn{2}{|c|}{ $l=k$, odd}\\
            \hline
            $ r $  \phantom{xxxxxxxxxxxxxxxx} & colors of $N(v_{k}^r)$  \phantom{xxx} \\
            \hline
            \hline
            $r\notin \{ 1,k-1 \}$ & $\{ r,1,k-1\}$  \\
            \hline
            $1$ & $\{ 1,2,3\} $ \\
            \hline
            $k-1$  & $\{ k-3,k-2,k-1\} $\\
            \hline
        \end{tabular}
        \label{fig.tabla_colorings}
\vspace{.2cm}
\begin{center}
Table 1: Colors of the neighborhoods of non-leaves of the comb $B_{k(k-1)}$.
\end{center}
\vspace{.7cm}

\newpage
\vspace{.2cm}
\begin{prop}\label{pro.tightuniciclico}
For every $k\ge 5$, there is a unicyclic graph $U_k$ with  NLC-number $\chi_{_{NL}}(U_k)=k$ and order $n(U_k)=2a_1(k)+a_2(k)$.
\end{prop}
\begin{proof}
Consider the $k$-NL-coloring of the cycle $C_{a_2(k)}$ obtained in the proof of Lemma~\ref{pro.1pairedcoloring}, that is, with all vertices having color-degree 2. 
There is an edge $xy$ with its endpoints $x$ and $y$ colored respectively with $2$ and $k-1$.
Consider the $k$-NL-coloring of the comb $B_{k(k-1)}$ given in the proof of Proposition~\ref{pro.colorcomb}.
Let $x'$ and $y'$ be  the vertices of degree $2$ of the comb $B_{k(k-1)}$ colored with $k-1$ and $2$, respectively.
Consider the unicyclic graph $U_k$ obtained from the union of the cycle and the comb, deleting the edge $xy$ from  the cycle $C_{a_2(k)}$ and adding the edges $xx'$ and $yy'$.
Notice that   $V(U_k)=V(C_{a_2(k)})\cup V(B_{k(k-1)})$, and thus the order of $U_k$ is $n(U_k)=n(C_{a_2(k)})+n(B_{k(k-1)})=a_2(k)+2k(k-1)=2a_1(k)+a_2(k)$ (see in Figure~\ref{uniciclic6} the case $k=6$).

\vspace{.1cm}
We claim that the $k$-NL-colorings of the cycle and the comb induce a $k$-NL-coloring in $U_k$.
We  have only changed the colors of the neighborhoods of $x'$ and $y'$.
On the one hand $x'$ has color $k-1$ and the colors of its neighbors are $\{ 1, 2, k-2 \}$. 
On the other hand, $y$ has color $2$ and the colors of its neighbors are $\{ 3,k-1,k\}$ if $k$ is even, and $\{ 1,k-1,k\}$, if $k$ is odd. 
We can check in the tables given in the proof of Proposition~\ref{pro.colorcomb} that any other vertex of the comb $B_{k(k-1)}$ has different color or different set of colors in their neighborhoods from those of $x'$ and $y'$. 
Hence, we have a $k$-NL-coloring of $U_k$.
\end{proof}


\begin{figure}[!hbt]
    \begin{center}
        \includegraphics[width=0.84\textwidth,page=8]{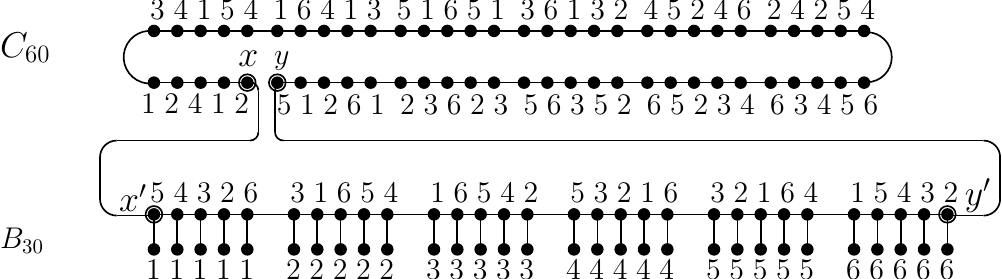}
        \caption{A 6-NL-coloring of  the  unicyclic graph $U_6$.}\label{uniciclic6}
    \end{center}
\end{figure}

\vspace{.2cm}
\begin{cor}
For every $k\ge 5$, the bound given in Theorem~\ref{unicyclic} is tight.
\end{cor}


\section{Trees}\label{s.trees}

In this section, we give some bounds for trees.

\begin{theorem}\label{trees}
Let $T$ be a non-trivial  tree. 
If  $\chi _{_{NL}}(T)=k\ge3$, then
    $$\displaystyle{n(G)\le 2a_1(k)+a_2(k)-2= \frac 12 (k^3+k^2-2k-4)}.$$
Moreover, if the equality holds, then $T$ has maximum degree $3$ and it contains  $k(k-1)$ leaves,  $\frac{k(k-1)(k-2)}{2}$ vertices of degree 2, and $k(k-1)-2$ vertices of degree 3.
    \end{theorem}
\begin{proof}
   Let $n$, $n_1$, $n_2$ and $n_{_{\ge 3}}$ be respectively the order, the number of leaves, the number of vertices of degree $2$ and the number of vertices of degree at least $3$ of a tree $T$.
    On the one hand, we know that
    \begin{align*}
    n_1+2\, n_2+ \sum_{\deg(u)\ge 3} \deg (u)= \sum_{u\in V(T)} \deg (u)=2 | E(T) | =2 (n-1)=2 (n_1+n_2+ n_{_{\ge 3}} -1).
    \end{align*}
    From here, we deduce that
    \begin{align}\label{n3}
    n_1=\sum_{\deg (u)\ge 3}{(\deg(u)-2)}+2\ge n_{_{\ge 3}} + 2.
    \end{align}
    On the other hand,
     $\chi _{_{NL}}(G)=k$ implies
    $n_1\le k(k-1)$ and  $n_2\le k \binom{k-1}{2} $.
    Therefore,
    \begin{align*}n&=n_1+n_2+n_{_{\ge 3}}\\
    &\le  k\Big( (k-1)+\binom{k-1}{2}\Big) + (n_1-2)\\
    &\le  k\Big( (k-1)+\binom{k-1}{2}\Big) + k(k-1)-2\\
    & =2a_1(k)+a_2(k)-2 \\
    &= \frac 12 (k^3+k^2-2k-4).
    \end{align*}

    Next, assume that there is a tree attaining this bound. In such a case, the inequalities in the preceding expression must be equalities.
    Thus, $n_{_{\ge 3}}=n_1-2=k(k-1)-2$, 
     $n_1=k(k-1)$, 
      and  $n_2=k\binom{k-1}{2}=\frac{k(k-1)(k-2)}{2}$.
    Finally, from Inequality~\ref{n3}, we deduce that $n_{_{\ge 3}}=n_1-2$ if and only if there are no vertices of degree greater than $3$.
    Therefore, there are exactly $n_{_{\ge 3}}$ vertices of degree $3$. Since $n_{_{\ge 3}}=n_1-2=k(k-1)-2$, the proof is complete.
\end{proof}

\begin{center}
    \begin{tabular}{c||c|c|c||ccc}
    $\chi_{_{NL}}(G)$  &  general graphs & $\Delta(G) =2$&  trees && trees & \\   
              \hline    
        $k$&$k\, (2^{k-1}-1)$&$\frac 12 (k^3-k^2)$&$\frac{1}{2} (k^3+k^2-2k-4)$&$n_1$& $n_2$ &$n_3$\\
        \hline
        3&9&9& \textbf{ 13}& \textbf{ 6}& \textbf{ 3 }& \textbf{ 4  } \\
        4&28&24&\textbf{ 34}& \textbf{ 12 }& \textbf{ 12 }& \textbf{ 10 }  \\
        5& 75 & 50& 68&20 &30 &18  \\
        6& 186& 90& 118&30 &60 &28   \\
        7& 441   & 147& 187&42 &105 &40\\
    \end{tabular}

Table 2: 
Upper bounds on the order of a graph for some values of $\chi _{_{NL}}(G)$.
\end{center}

Table 2 illustrates Theorem~\ref{trees}.
The cases in bold are not feasible because the bound for general graphs (see Theorem~\ref{gb}) is smaller than  the specific bound for trees (see Theorem~\ref{trees}).
The bound for graphs with $\Delta =2$ is given in Proposition~\ref{pro.verticeslk}.
Bounds for unicyclic graphs (see Theorem~\ref{unicyclic}) are the ones for trees adding two unities.  
The last column shows the number of vertices of degree 1 ($n_1$), of degree 2 ($n_2$) and of degree 3 ($n_3$) that a tree attaining the upper bound has to have, as shown  in Theorem~\ref{trees}.

\vspace{.2cm}
For $k=3$, the path $P_9$ is an example attaining the general upper bound.
For $k=4$,  a tree attaining the general upper bound $n=28$ is displayed in Figure~\ref{fig.22.66}.
For $k=5$, a tree of order $66$ is shown in  Figure~\ref{fig.22.66}.
We do not know whether  there are trees of order either $67$ or $68$ with NLC-number 5.
Next proposition shows that  there is a tree attaining the specific upper bound for trees whenever $k\ge 6$.

\begin{figure}[!hbt]
    \begin{center}
        \includegraphics[width=0.41\textwidth,page=9]{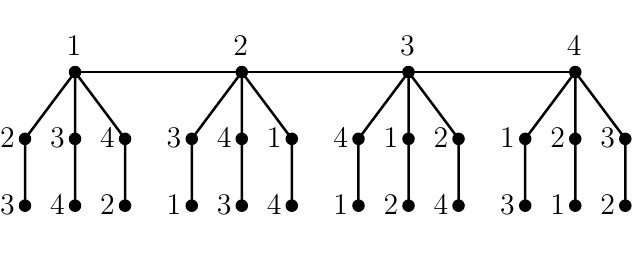}
        \hspace{4mm}
        \includegraphics[width=0.53\textwidth,page=10]{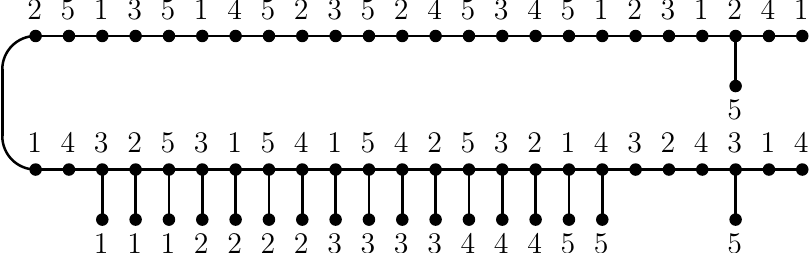}
        \caption{A tree $T_1$ of order $28$ and $\chi_{NL} (T_1)=4$ (left) and
            a tree $T_2$ of order $66$ and $\chi_{NL} (T_2)=5$ (right).}\label{fig.22.66}
    \end{center}
\end{figure}

Recall that a \emph{caterpillar} is a tree that reduces to a path when pruning all its leaves. 
Clearly, any comb is a caterpillar.

\begin{prop}\label{treess}
For every integer $k\ge 6$, there is a caterpillar $T$  with NLC-number $\chi_{NL} (T)=k$ and  order 
$\displaystyle{n(T)= \frac 12 (k^3+k^2-2k-4)}$.
\end{prop}
\begin{proof}
    Consider the $k$-NL-coloring of the unicyclic graph $U_k$ of order $\frac 12 k(k-1)(k+2)$ described in Proposition~\ref{pro.tightuniciclico}. 
    Consider the leaf of color 2 hanging from the vertex $x$ colored with $k-1$.
    Delete both vertices and add the edge joining the remaining neighbors of $x$.
    Do the same with the leaf colored with $k-1$ hanging from the vertex $y$ of color $2$.
    Remove the edge joining the vertex $u$ of degree 2 and color $2$ with the vertex $v$ of degree 3 and color $k-1$.
    Attach a leaf colored with $k-1$ to vertex $u$ and a leaf colored with $2$ to vertex $v$.
    We obtain a tree $T_k$ of order $\frac 12 k(k-1)(k+2)-4+2=\frac 12 k(k-1)(k+2)-2$
    (see an example in Figure~\ref{tree6}).
    
    \vspace{.2cm}
    We claim that in such a way we have a $k$-NL-coloring of the tree $T_k$.
    Indeed, we have only  changed the colors of the neighborhoods of the vertices adjacent to $x$ and to $y$ in $T_k$.
    Following the notations of the proof of Proposition~\ref{pro.colorcomb}, we have $x=v_{k-1}^2$ and $y=v_{2}^{k-1}$, and, if $k\ge 6$,
    the vertices adjacent to them in $M_2$ and $M_{k-1}$ are respectively
    $v_{k-2}^2$, $v_{k}^2$ and $v_{1}^{k-1}$, $v_{3}^{k-1}$.
    After deleting the vertices $x$ and $y$ from the comb, the colors of their neighborhoods are given in  Table 3.

\vspace{.5cm}

    \noindent
    \begin{tabular}{r|c|c|c|c}
        $z$
        &
        $v_{1}^{k-1}$
        &
        $v_{3}^{k-1}$
        &
        $v_{k-2}^2$
        &
        $v_{k}^2$
        \\
        \hline
        color of $z$ & 1 & 3 & $k-2$ & $k$
        \\
        \hline
        colors of $N(z)$ in $T_k$
        &
        $\{ 3, k-1,k\}$
        &
        $\{ 1,4, k-1\}$
        &
        $\begin{array}{lc} \{ 2, k-3,k\},& \hbox{ if }k\ge 7 \\ \{ 1,2,6\},&\hbox{ if }k= 6 \end{array} $
        &
        $\{ 1,2, k-2\}$
        \\
    \end{tabular}
\vspace{.2cm}
\begin{center}
Table 3: Set of colors of the neighborhoods.
\end{center}
\vspace{.5cm}

    Using  Table 1, check that there are no vertices with the same color having the same set of colors in their neighborhoods in $T_k$. Therefore,  we have a $k$-NL-coloring of $T_k$.
\end{proof}

\vspace{.2cm}
\begin{figure}[!hbt]
    \begin{center}
        \includegraphics[width=0.86\textwidth,page=11]{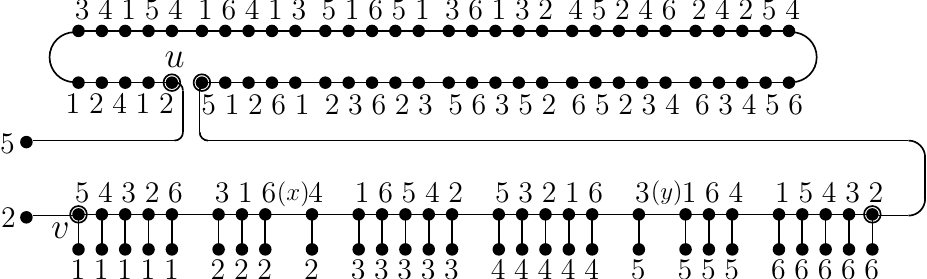}
        \caption{A 6-NL-coloring of  a tree of order  118 constructed from a 6-NL-coloring of a unicyclic graph of order $120$.}\label{tree6}
    \end{center}
\end{figure}

\vspace{.2cm}
Finally, some others results involving the NLC-number of trees are shown.

\vspace{.2cm}
\begin{prop}
Let  $T$ be a tree  of   order  $n(T)=n\geq 5$. 
If $T$ is a star, then $\chi _{_{NL}}(T)=n(T)$; otherwise $\chi _{_{NL}}(T)\leq n(T)-2$.
\end{prop}
\begin{proof}
If ${\rm diam}(T)=2$, then $T$ is a star and thus $\chi _{_{NL}}(T)=n$ (see \cite{xnl1}).

If ${\rm diam}(T)=3$, then $T$ is a double star, that is, $T$ has exactly two adjacent vertices $u$ and $v$ which are not leaves; and $u$ is adjacent  to $r$ leaves and $v$ is adjacent  to $s$ leaves, $1\leq r\leq s\leq n-3$. 
Hence, $\chi _{_{NL}}(T)=s+1\leq n-2$ (see \cite{cherheslzh02} and check that the coloring given in Proposition 4.1 is  an NL-coloring).

If ${\rm diam}(T)\geq 4$, then consider a pair of  vertices $x, y$ at distance $4$, three   vertices $a,b,c$ such that $xa, ab, bc, cy$ are edges of $T$ and the following $(n-2)$-NL-coloring: the same color for $x, b$ and $y$, and a different color for every other vertex of $T$. 
Thus, $\chi _{_{NL}}(T)\leq n-2$.
\end{proof}

\vspace{.2cm}
\begin{prop}\label{maximum.degree}
Let $T$ be a tree. If $\chi _{_{NL}}(T)=k$, then $\Delta (T)\le (k-1)^2+\frac{k-1}2$.
\end{prop}
\begin{proof}
    Suppose to the contrary that $\Delta (T) > (k-1)^2+\frac{k-1}2$.
    If $u$ is a vertex of degree $\Delta$ and color $k$, then its neighbors have colors $1,\dots ,k-1$. 
    It can be easily proved that $u$ has at most $(k-1)^2$ neighbors of degree at most 2. 
    Hence, $u$ must have $>\frac{k-1}2$ neighbors of degree at least 3. 
    Thus, the number $n_1$ of leaves  satisfies $n_1>(k-1)^2+2\frac{k-1}2=k(k-1)$. 
    But $T$ has at most $k(k-1)$ leaves, getting a contradiction.
\end{proof}

\section{Conclusions and open problems}\label{op}

We have determined the NLC-number of  paths, cycles (see Theorem~\ref{pro.base} and Theorem~\ref{thm.chi_CiclosCaminos}), 
and also of  fans and wheels (see Theorem~\ref{t:fan_wheel_small} and Theorem~\ref{t:fan_wheel}).

In \cite{xnl1}, the order of a graph is bounded from above by a function of its NLC-number.
In the present paper, we have achieved better bounds both for unicyclic graphs (see Theorem~\ref{unicyclic}) and for trees (see Theorem~\ref{trees}).
Moreover, we have shown that these new bounds are tight for NLC-number $k\geq 5$ in the case of unicyclic graphs (see Proposition~\ref{pro.tightuniciclico}); and for NLC-number $k\geq 6$ in  the case of trees (see Proposition~\ref{treess}). 
In this last case, we have proved that the bound is achieved by a caterpillar.
For trees with NLC-number $k\geq 5$, the maximum order according to our bound is $68$, but the maximum order that we have obtained is $66$. 
The existence or not  of trees with NLC-number $5$ and order $67$ or $68$ remains open.

According to Theorem \ref{gb}(2), if $\chi _{_{NL}}(G)=k$ and $\Delta(G)\le k-1$, then $\displaystyle{n(G)\le \, \sum_{j=1}^{\Delta(G) } a_j(k)}.$
In general, it is not known whether or not this bound is tight. 
For $\Delta(G) = k-1$, this bound match with the one given in Theorem~\ref{gb}(1) which is known to be tight.
Since we have proved that if $G$ is a cycle, then $n= a_1(k)+ a_2(k)$, this fact  implies that the referred bound is tight for graphs with maximum degree $\Delta(G)=2$.
What does it happen for graphs with $\Delta(G)=3$? 
We have shown that if $G$ is either a tree or an unicyclic graph, then $n\leq 2a_1(k)+a_2(k)$ and, obviously, $2a_1(k)+a_2(k)< a_1(k)+a_2(k)+a_3(k)$, so the bound is not achieved by these graphs.
Is it achieved by other kind of graphs? 
This is an open problem, not only for graphs with maximum degree $\Delta=3$,  but also  for graphs with  $4\leq \Delta \leq k-2$.

We know that  the bound given by the Proposition~\ref{maximum.degree} is not tight. 
Indeed, for a tree $T$ with NLC-number  $3$,   Proposition~\ref{maximum.degree} states that $\Delta(T)\leq 5$. 
However, by  Theorem~\ref{gb}(1),  we have that $n(T)\leq 9$. 
Then, it is easy to verify that $\chi _{_{NL}}(T)=3$ and $n(T)\leq 9$ implies $\Delta(T)\leq 4$.

In general, we postulate  the following.

\begin{conj}\label{treedelta} Let $k\geq 2$.
If $T$ is a tree with $\chi _{_{NL}}(T)=k$, then $\Delta(T) \le (k-1)^2$.
Moreover, this bound is tight for every integer $k\ge 2$.
\end{conj}

For an example of a tree $T$ with $\chi _{_{NL}}(T)=k$ and $\Delta(T) = (k-1)^2$, see Figure~\ref{xnlgr}.

\begin{figure}[!hbt]
\begin{center}
\includegraphics[width=0.75\textwidth,page=12]{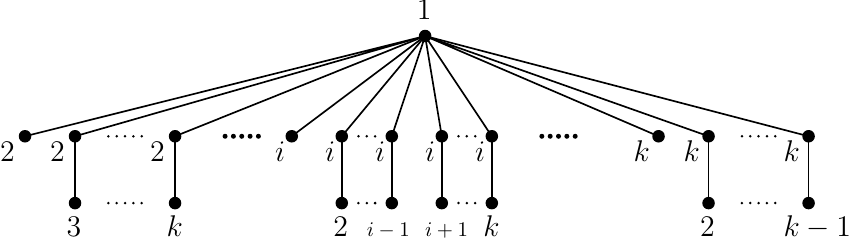}
\caption{ Tree $T$ with $\Delta(T)=(k-1)^2$ and  $\chi _{_{NL}}(T)=k$.}\label{xnlgr}
\end{center}
\end{figure}

If $G$ is a connected graph of diameter ${\rm diam}(G)=d\leq 23$, then, it is possible  verify 
that $\chi _{_{NL}} (G) \ge \chi _{_{NL}} (P_{d+1})$. 
We propose the following conjecture.

\begin{conj}
    Let $G$ be a graph of diameter $d$.
    Then, $\chi _{_{NL}} (G) \ge \chi _{_{NL}} (P_{d+1})$.
\end{conj}


\end{document}